\documentclass[reqno, 11pt]{amsart}
\usepackage{amsfonts}
\usepackage{amsmath}
\usepackage{tikz-cd}
\usepackage{mathtools}
\usepackage{amssymb}
\usepackage{amsthm}
\usepackage{amscd}
\usepackage{color}
\usepackage{verbatim}

\usepackage[colorlinks]{hyperref}
\definecolor{linkblue}{RGB}{1,1,190}
\definecolor{citered}{RGB}{190,1,1}  
\hypersetup{
	linkcolor=linkblue,
	urlcolor=linkblue,
	citecolor=citered
}

\newtheorem{theorem}{Theorem}
\newtheorem{lemma}[theorem]{Lemma}

\newtheorem{proposition}[theorem]{Proposition}

\newcommand{\N}{\mathbb N}

\DeclareMathOperator{\Pic}{Pic} 
 
 \DeclareMathOperator{\Cl}{Cl}
\DeclareMathOperator{\spec}{Spec} \DeclareMathOperator{\ord}{ord}
 
\DeclareMathOperator{\Reg}{Reg}

\def\p{\ensuremath\mathfrak{p}}
\def\O{\ensuremath\mathcal{O}}
\def\f{\ensuremath\mathfrak{f}}

\begin{document}
\title[A Characterization of Transfer Krull Orders in Dedekind Domains]{A Characterization of Transfer Krull Orders in Dedekind Domains with Torsion Class Group}
\author{Balint Rago}
\address{University of Graz, NAWI Graz, Department of Mathematics and Scientific Computing, Heinrichstraße 36,
8010 Graz, Austria}

\thanks{This work was supported by the Austrian Science Fund FWF, Project Number W1230}

\email{balint.rago@uni-graz.at}



\begin{abstract}
We establish a characterization (under some natural conditions) of those orders in Dedekind domains which allow a transfer  homomorphism to a monoid of zero-sum sequences. As a consequence, the inclusion map to the Dedekind domain is a transfer homomorphism, with the exception of a particular case. The arithmetic of Krull and Dedekind domains is well understood, and the existence of a transfer homomorphism implies that the order and the associated Dedekind domain share the same arithmetic properties. This is not the case for arbitrary orders in Dedekind domains. 
\end{abstract}

\maketitle

\section{Introduction}

A (commutative integral) domain is factorial if and only if it is Krull with trivial class group. Every Krull monoid (in particular, every Krull and every Dedekind domain) has a transfer homomorphism to the monoid of zero-sum sequences over the subset of classes of its class group that contain prime divisors.
Transfer homomorphisms allow to pull back arithmetic invariants (such as sets of lengths) from the target monoid back to the original object of interest. Thus, the existence of transfer homomorphisms justifies the classic philosophy that the arithmetic of rings of integers in algebraic number fields and of Dedekind domains depends only on their class group. Monoids of zero-sum sequences are Krull monoids, which are usually studied with combinatorial methods (\cite{Ge-Ru09, Gr22a}). If the class group is finite, then additive combinatorics plays a central role. It provides precise results on sets of lengths (and on invariants controlling their structure, such as elasticities and sets of distances) in terms of the group invariants.

Factorization theory studies the arithmetic of a wide range of domains, including weakly Krull domains, monoid algebras, and rings of integer valued polynomials.  Orders in Dedekind domains are the simplest class of non-Krull domains in this scenario. They are weakly Krull, whence they admit transfer homomorphisms to $T$-block monoids (constructed from zero-sum sequences and localizations at primes containing the conductor). 
They allow to obtain arithmetic finiteness results, but in general these results are less precise than those obtained for Krull domains (see, for example, \cite[Theorem 3.7.1]{Ge-HK06a}). But not only are the methods weaker but the arithmetic of general non-principal orders shows different features compared with the arithmetic of the corresponding principal orders (which are Dedekind and hence Krull). To give an example, the set of distances of a ring of integers (whose class group has at least three elements) is an interval with minimum one. This property does not even hold for arbitrary orders in quadratic number fields (\cite{Re23a}).

In his seminal paper \cite{Sm13a}, Smertnig characterized the maximal orders in central simple algebras over algebraic number fields which allow a transfer homomorphism to a monoid of zero-sum sequences. The characterization shows that, apart from an exceptional case, these maximal orders (which are non-commutative Dedekind domains) allow such a transfer homomorphism. His work initiated the search for rings and monoids, that are not Krull, but nevertheless allow a transfer homomorphism to a monoid of zero-sum sequences (or equivalently, to a commutative Krull monoid). 
They are called transfer Krull and include non-commutative rings, Noetherian domains that are not integrally closed, and others (e.g, \cite{Sm19a, Ge-Sc-Zh17b, Ba-Sm22a, Ba-Re22a}). At the same time, other classes of rings and monoids were revealed that do not allow such transfer homomorphisms. They include power monoids, rings of integer-valued polynomials, and further classes of commutative and non-commutative rings (e.g., \cite{Fa-Tr18a, Fr-Na-Ri19a, B-B-N-S23a}).  For a list of transfer Krull and non-transfer Krull monoids and domains,  we refer to the survey \cite{Ge-Zh20a}.

The arithmetic of orders in Dedekind domains (with a focus on orders in algebraic number fields) has been studied since decades and found a renewed interest in recent years (e.g., \cite{Br-Ge-Re20, Ge-Ka-Re15a, Ph12b, Pi00, G-L-T-Z21, Po-24, Po-24a, Ra-23}).
In the present paper, we establish an algebraic characterization of transfer Krull orders in Dedekind domains. The mild conditions, which we impose, are satisfied by  rings of integers in algebraic number fields (\cite[Corollary 4.4.4]{HK20a}), by  holomorphy domains in global fields and associated $S$-class groups (\cite[Corollary 7.9.3]{HK22a}), and by other, more abstract classes of Dedekind domains (see for example  \cite{Ch-Ge24a} and the references given there). 

\begin{theorem} \label{thm}
  Let $R$ be a Dedekind domain with torsion class group and let $\mathcal O \subseteq R$ be an order with conductor $\mathfrak f$ such that $R/\mathfrak f$ is finite. 
  \begin{enumerate}
  \item Suppose that $|\Cl (R)| \ne 2$ and that every class of $\Cl (R)$ contains a prime ideal coprime to $\mathfrak f$. Then $\mathcal O$ is transfer Krull if and only if the following two conditions hold. \smallskip
      \begin{itemize}
      \item[(a)] $\mathcal O\cdot R^{\times} = R$. \smallskip
      
      \item[(b)] $\mathsf v_{\mathfrak p} (u)=1$ for all atoms $u \in \mathcal O_{\mathfrak p}$ and all $\mathfrak p \in \spec ( \mathcal O)$. 
      \end{itemize} 

\smallskip
  \noindent Moreover, if $\mathcal O$ is transfer Krull, then the inclusion $\mathcal O^{\bullet} \hookrightarrow R^{\bullet}$ is a transfer homomorphism and there is a transfer homomorphism $\theta \colon \mathcal O^{\bullet} \to \mathcal B (\Cl (R) )$, where $\mathcal B (\Cl (R) )$ is the monoid of zero-sum sequences over the class group $\Cl (R)$. \\
  
  \item Suppose that $|\Cl (R)| = 2$ and that every class of $\Pic (\mathcal O)$ contains a prime ideal coprime to $\mathfrak f$. Then $\mathcal O$ is transfer Krull if and only if the following two conditions hold. \smallskip
      \begin{itemize}
      \item[(a)] $\mathcal O \cdot R^{\times} = R$. \smallskip
      
      \item[(b)] $\mathsf v_{\mathfrak p} (u) \in \{1,2\}$ for all atoms $u \in \mathcal O_{\mathfrak p}$ and all $\mathfrak p \in \spec ( \mathcal O)$ and if the prime ideal lying over $\p$ is principal, then $\mathsf v_{\mathfrak p} (u) =1$ for all atoms $u\in\O_\p$.
    
      \end{itemize}
  \end{enumerate}
 
\end{theorem}

\smallskip
Theorem \ref{thm} implies that the arithmetic results derived for $R$ hold true for a transfer Krull order $\mathcal O \subseteq R$ as well. If $\Cl (R)$ is finite, then  a survey of Schmid \cite{Sc16a} provides a good overview of results achieved with methods from additive combinatorics. If $\Cl (R)$ is infinite, then every finite nonempty subset of $\N_{\ge 2}$ occurs as a set of lengths of $R$ (\cite[Theorem 7.4.1]{Ge-HK06a}), whence also as a set of lengths of $\mathcal O$.

\smallskip 
The characterization in the case $|\Cl(R)|=2$ is a generalization of the result in \cite{Ra-23}, where half-factorial orders in algebraic number fields are studied. We discuss the particularities of this case in a final remark at the end of Section 4. \\

We will proceed as follows. In Section 2, we recall several definitions and concepts from factorization theory, in particular the notion of transfer Krull. Moreover, we present some algebraic properties of Dedekind domains and orders. \\

In Section 3, we study the relationship of regular elements of transfer Krull orders with the monoid $\Reg_\f(R)$ and we prove that $\Pic(\O)\simeq \Cl(R)$, except possibly in the special case $|\Cl(R)|\neq 2$. \\

In the final section, we show that the Spec-map is bijective and then turn our attention to the irregular localizations $\O_\p$, before proving Theorem \ref{thm}.

\section{Notation and Preliminaries}

\subsection[short]{Monoids and factorization theory}

Let $H$ be a commutative semigroup with identity. $H$ is called \textit{cancellative} if for all $a,b,c\in H$ with $ac=bc$, it follows that $a=b$. \bigskip

\begin{center}
    \textit{Throughout this paper, a monoid is always a commutative, cancellative semigroup with identity.} 
\end{center}

\bigskip
We denote by $\mathbb{N}$ the set of positive integers and by $\mathbb{N}_0$ the set of non-negative integers. For $n,m\in\mathbb{N}_0$, we write $[n,m]=\{x\in\mathbb{N}_0 \ | \ n\leq x\leq m\}$.
Let $H$ be a multiplicatively written monoid with identity $1_H$. We let $H^\times$ denote the \textit{group of units} (invertible elements) of $H$ and by $H_{\text{red}}=H/H^\times=\{aH^\times \ |\  a\in H\}$ the associated reduced monoid. In general, we call a monoid $H$ \textit{reduced} if $H^\times=\{1\}$. A subset $T\subseteq H$ is called a \textit{submonoid} of $H$ if $1_H\in T$ and $ab\in T$ for all $a,b\in T$. We let $\mathsf{q}(H)=\{ab^{-1}\ | \ a,b\in H\}$ denote the \textit{quotient group} of $H$. This group is uniquely determined up to canonical isomorphism. Any monoid $T$ between $H$ and $\mathsf{q}(H)$ is called an \textit{overmonoid} of $H$. For $a,b\in H$, we set $a\sim_H b$ if $a=\varepsilon b$ for some $\varepsilon\in H^\times$ and we say that $a$ and $b$ are \textit{associated} in $H$. \\

Let $a$ be a non-unit of $H$. Then $a$ is called an \textit{atom} if $a$ is not the product of two non-units of $H$ and $a$ is said to be \textit{prime} if, whenever $a$ divides $bc$ for elements $b,c\in H$, then $a$ divides $b$ or $c$. Moreover, we say that an atom $a$ is \textit{absolutely irreducible} if, for every $n\in\mathbb{N}$, $a^n$ has precisely one factorization into atoms (up to associates). We denote by $\mathcal{A}(H)$ the set of atoms of $H$ and say that $H$ is \textit{atomic} if every non-unit can be written as a finite product of atoms. For each non-unit $a\in H$, we let $\mathsf{L}_H(a)=\{k\in\mathbb{N}\ | \ a \text{ is a product of } k \text{ atoms of } H  \}$ be the \textit{set of lengths of} $a$ and we set $\mathsf{L}_H(a)=\{0\}$ if $a\in H^\times$. An atomic monoid is said to be \textit{factorial} if every element of $H$ has precisely one factorization up to associates, or equivalently, if every atom of $H$ is prime. 
Moreover, $H$ is called \textit{half-factorial} if $|\mathsf{L}_H(a)|=1$ for every $a\in H$. Clearly, every factorial monoid is half-factorial. Let $R$ be an integral domain. We will denote the multiplicative monoid of non-zero elements of $R$ by $R^\bullet$. All arithmetic terms defined for monoids carry over to domains and we set $\mathcal{A}(R)=\mathcal{A}(R^\bullet)$ and so on.
\\

Let $G$ be an additively written, abelian group and let $\mathcal{F}(G)$ be the (multiplicative) free abelian monoid over $G$. For not necessarily distinct elements $g_1,\ldots,g_n\in G$, we denote by $\sigma(g_1\cdot\ldots\cdot g_n)=g_1+\ldots+g_n\in G$ the \textit{sum} of the sequence $g_1\cdot\ldots\cdot g_n\in\mathcal{F}(G)$. Moreover, we denote the monoid of all sequences $x\in\mathcal{F}(G)$ with $\sigma(x)=0$ by $\mathcal{B}(G)$ and we call $\mathcal{B}(G)$ the \textit{monoid of zero-sum sequences} over $G$. \\

We denote by \[\widehat{H}=\{x\in\mathsf{q}(H) \ | \ \text{there exists some } c\in H, \text{ such that } cx^n\in H \text{ for all } n\in\mathbb{N} \}\] the \textit{complete integral closure} of the monoid $H$ and we say that $H$ is completely integrally closed if $H=\widehat{H}$. A monoid $H$ is called \textit{Krull} if $H$ is completely integrally closed and satisfies the ascending chain condition for $v$-ideals. For the definition of $v$-ideals and a detailed treatment of Krull monoids, we refer to \cite{Ge-HK06a}. If $R$ is an integral domain, then $R^\bullet$ is a Krull monoid if and only if $R$ is a Krull domain. Every Krull monoid $H$ has a \textit{class group} $\mathcal{C}(H)$ associated to it.  \\

Let $H$ and $B$ be monoids. A monoid homomorphism $\theta:H\to B$ is said to be a \textit{transfer homomorphism} if the following two properties are satisfied. \begin{itemize}
  \item[\textbf{(T1)}] $B=\theta(H)\cdot B^\times$ and $\theta^{-1}(B^\times)=H^\times$
  \item[\textbf{(T2)}] If $u\in H$, $b,c\in B$ and $\theta(u)=bc$, then there exist $v,w\in H$ such that $u=vw$, $\theta(v)\sim_B b$ and $\theta(w)\sim_B c$.
\end{itemize} 
Henceforth, we will refer to property \textbf{(T2)} as the \textit{transfer property}. Transfer homomorphisms preserve many arithmetic invariants, most importantly, sets of lengths. We have $\mathsf{L}_H(u)=\mathsf{L}_B(\theta(u))$ for all $u\in H$. If $H$ is a Krull monoid with the additional property that every class in $\mathcal{C}(H)$ contains a prime divisor, then $H$ admits a transfer homomorphism $\boldsymbol{\beta}:H\to\mathcal{B}(\mathcal{C}(H))$. \\

A monoid $H$ is said to be \textit{transfer Krull} if there exists a Krull monoid $B$ and a transfer homomorphism $\theta:H\to B$. Since the identity is a transfer homomorphism, every Krull monoid is transfer Krull, but the converse does not necessarily hold. For example, every half-factorial monoid $H$ is transfer Krull via the transfer homomorphism $\theta:H\to (\mathbb{N}_0,+)$, which maps every element to its unique factorization length, but not every half-factorial monoid is Krull. Thus, transfer homomorphisms are a useful tool to pull back arithmetic properties of a Krull monoid, whose arithmetic is very well understood, to $H$. The following proposition summarizes some properties of transfer Krull monoids.  

    \begin{proposition}\label{tkrull} 
        Let $S$ be a monoid with quotient group $\mathsf{q}(S)$.
       Then $S$ is transfer Krull if and only if there is a Krull monoid $H$ with $S\subseteq H\subseteq \mathsf{q}(S)$, such that the inclusion $S\hookrightarrow H$ is a transfer homomorphism. If this holds, then we have $H=SH^\times$, $H^\times\cap S=S^\times$ and $\mathcal{A}(H)=\{u\varepsilon \ | \ u\in\mathcal{A}(S),\varepsilon\in H^\times\}$.
    \end{proposition}
    \begin{proof}
      See \cite[Proposition 5.3]{Ge-Zh20a} and \cite[Proposition 2.7]{Ba-Re22a}
    \end{proof}

A monoid $H$ is called \textit{finitely primary} if there exist positive integers $\alpha,s$ with the following properties: 
\begin{enumerate}
    \item $H$ is a submonoid of a factorial monoid $F$, generated by $s$ pairwise non-associated prime elements $p_1,\ldots, p_s$. 
    \item $H\setminus H^\times \subseteq p_1\cdots p_sF$. 
    \item $(p_1\cdots p_s)^\alpha F \subseteq H$.
\end{enumerate}

If this is the case, then we say that $H$ is finitely primary of \textit{rank} $s$ and \textit{exponent} $\alpha$. \\

Let $\varepsilon\in F^\times$ and let $x=\varepsilon p_1^{k_1}\cdots p_s^{k_s}\in H$ for non-negative integers $k_i$. For each $i\in [1,s]$, we define the function $\mathsf{v}_{p_i}:H\to\mathbb{N}_0$, where $\mathsf{v}_{p_i}(x)=k_i$.

\begin{lemma}\label{finp}
    If $H\subseteq F$ is finitely primary of rank $s\geq 2$ and exponent $\alpha$, then $|\mathsf{v}_{p_i}(\mathcal{A}(H))|=\infty$ for each prime element $p_i\in F$.
\end{lemma}

\begin{proof}
   Let $k$ be a positive integer and consider the element $x=p_i^{\alpha k}(p_1\cdots p_s)^\alpha$. By definition, we have $x\in H$. For every $y\in\mathcal{A}(H)$ and $j\in[1,s]$, we have $\mathsf{v}_{p_j}(y)\geq 1$ and thus $\max\mathsf{L}(x)\leq \alpha$. Hence, in every factorization of $x$, there is an atom $y$ satisfying $\mathsf{v}_{p_i}(y)\geq k+1$. Since $k$ was chosen arbitrarily, we obtain $|\mathsf{v}_{p_i}(\mathcal{A}(H))|=\infty$.
\end{proof}

\subsection[short]{Dedekind domains and orders} Let $R$ be a Dedekind domain with quotient field $K$. Since $R$ is a Krull domain, $R^\bullet$ is a Krull monoid and the class group $\mathcal{C}(R^\bullet)$ coincides with the usual ideal class group $\Cl(R)$. We denote the class of an ideal $I\subseteq R$ by $[I]$ and we write the class group additively, whence $[IJ]=[I]+[J]$ for ideals $I,J$. If every class in $\Cl(R)$ contains a prime ideal, then there is a transfer homomorphism $\boldsymbol{\beta}:R^\bullet\to\mathcal{B}(\Cl(R))$, defined in the following way. If $a\in R^\bullet$ with $aR=P_1^{n_1}\cdots P_k^{n_k}$ for prime ideals $P_1,\ldots,P_k\in\spec(R)$, then $\boldsymbol{\beta}(a)=[P_1]^{n_1}\cdot\ldots\cdot[P_k]^{n_k}$.
It is easy to see that an element $a\in R$ is an atom (equivalently, that $\boldsymbol{\beta}(a)$ is an atom) if and only if $aR=P_1^{n_1}\cdots P_k^{n_k}$ does not contain any proper non-empty subproduct that is principal. Moreover, if $\Cl(R)$ is a torsion group, then $a$ is absolutely irreducible if and only if $aR=P^{\ord([P])}$ for a prime ideal $P$ (for details see \cite[Proposition 7.1.5]{Ge-HK06a}) and for every element $b\in R$, there is $n\in\mathbb{N}$ such that $b^n$ can be written as a product of absolutely irreducible elements.\\ 

We will now gather several facts and properties about orders in Dedekind domains. For references, see for example \cite[2.10]{Ge-HK06a} and \cite[2.11]{HK20a}. \\

An \textit{order} in a Dedekind domain $R$ is a subring $\O\subseteq R$ with quotient field $\mathsf{q}(\O)=\mathsf{q}(R)=K$ such that $R$ is a finitely generated $\O$-module. From now on, we will always assume that $\O$ is a proper subring of $R$. Every order is one-dimensional and Noetherian but not integrally closed. In particular, $\O^\bullet$ cannot be Krull. The \textit{conductor} of $\O$ is the greatest $R$-ideal $\f$ that is contained in $\O$. The conductor is always a nonzero ideal and we have \[\f=\{x\in R \ | \ xR\subseteq\O\}.\] We will denote the set of $R$-prime ideals coprime to $\f$ by $\spec_\f(R)$.\\

Let $\p\in \spec(\O)$. Then $\O_\p$ is a DVR if and only if $\p\not\supseteq \f$. The integral closure $\overline{\O_\p}$ is a semilocal PID with $n$ maximal ideals, where $n$ is the number of $R$-prime ideals lying over $\p$. Moreover, $\O_\p^\bullet$ is a finitely primary monoid of rank $n$. Let $p$ be a prime element of $\overline{\O_\p}$ and let $a\in\overline{\O_\p}$. The element $p$ naturally corresponds to a prime ideal $P\in\spec(R)$.
We denote by $\mathsf{v}_p(a)$ the highest power of $p$ that divides $a$ and if $a\in R$, then $\mathsf{v}_p(a)=\mathsf{v}_P(a)$, where $\mathsf{v}_P(a)$ denotes the highest power of $P$ that divides $aR$. If $a\in \O_\p$ and $x\in (\O_\p^\bullet)_{\text{red}}$ such that $x=a\O_\p^\times$, then we write $\mathsf{v}_p(a)=\mathsf{v}_p(x)$.  If there is only one $R$-prime ideal lying over $\p$, then $\overline{\O_\p}$ is a DVR with precisely one prime element up to associates. In this case we simply write $\mathsf{v}_\p(a)$. \\

An element $a\in\O$ is said to be \textit{regular} if $a+\f\in(\O/\f)^\times$. We denote the (multiplicative) monoid of regular elements of $\O$ by $\Reg(\O)$. Similarly, we define the monoid \[\Reg_\f(R):=\{a\in R \ | \ a+\f\in (R/\f)^\times\}.\] Note that, if $a\in\Reg(\O)$ and $b\in R$, then $ab\in \O$ implies $b\in\O$. Moreover, $\Reg(\O)$ is a divisor-closed submonoid of $\O^\bullet$, meaning that whenever $ab\in\Reg(\O)$ for elements $a,b\in\O^\bullet$, then $a,b\in\Reg(\O)$. 
Similarly, $\Reg_\f(R)$ is a divisor-closed submonoid of $R^\bullet$. We call the extension $\O\subseteq R$ a \textit{root extension} if for every $a\in R$ there is $n\in\mathbb{N}$ such that $a^n\in \O$. If $R/\f$ is a finite ring, then $\O\subseteq R$ is a root extension if and only if the map \[\spec(R)\to\spec(\O)\]\[P\mapsto P\cap\O\] is bijective, for a proof see \cite[Theorem 2]{Ka20a}. We will call this map the Spec-map. More precisely, if $R/\f$ is finite and $a\in R$, then $a^n\not\in\O$ for all $n\in\mathbb{N}$ if and only if there are $P,Q\in\spec(R)$, lying over the same prime ideal $\p$, such that $\mathsf{v}_P(a)\geq 1$ and $\mathsf{v}_Q(a)=0$. In particular, $\Reg(\O)\subseteq \Reg_\f(R)$ is a root extension of monoids.  \\

Let $\Pic(\O)$ denote the Picard group of $\O$. There is an exact sequence \[1\to R^\times/\O^\times\to(R/\f)^\times/(\O/\f)^\times\to\Pic(\O)\to\Cl(R)\to 0,\] which implies that we have $\Pic(\O)\simeq \Cl(R)$ if and only if $ R^\times/\O^\times\simeq(R/\f)^\times/(\O/\f)^\times$. Note that this is also equivalent to $\Reg(\O)\cdot R^\times=\Reg_\f(R)$. Moreover, we have $\O\cdot R^\times=R$ if and only if $\Pic(\O)\simeq \Cl(R)$, the Spec-map is bijective and $1\in\mathsf{v}_\p(\mathcal{A}(\O_\p))$ for all $\p\in\spec(\O)$. Note also that the canonical surjection $\Pic(\O)\to\Cl(R)$ implies that if every class of $\Pic(\O)$ contains a prime ideal, then the same holds true for $\Cl(R)$. In particular, our assumption in Theorem \ref{thm} (2) is stronger than the one made in (1). \\

Let $\p\in\spec(\O)$, let $P_1,\ldots,P_k$ be the $R$-prime ideals lying over $\p$ and let $p_1,\ldots,p_k$ be prime elements of $\overline{\O_\p}$ such that $p_i$ corresponds to $P_i$ for all $i\in [1,k]$. Moreover, let $a\in R$ and $x\in \O_\p$ such that $aR^\times\cap\O\neq\emptyset$ and $\mathsf{v}_{P_i}(a)=\mathsf{v}_{p_i}(x)$ for all $i\in [1,k]$. Then there is $b\in \Reg_\f(R)$ such that $ab\sim_{\O_\p}x$ and the choice of $b$ only depends on its class in $(R/\f)^\times/(\O/\f)^\times$. If we now suppose in addition that $\Pic(\O)\simeq\Cl(R)$, then we can choose $b$ to be a unit and hence we find an element $c\in\O$ such that \[a\sim_R c\sim_{\O_\p} x.\] We will use this fact several times in the final section. \\

Let $H$ be an overmonoid of $R^\bullet$, $a\in R$ and let $I,I_1,I_2$ be principal ideals of $R$. We write $a\sim \O$ if there is $b\in\O$ such that $a\sim_R b$, and $I\sim\O$ if there is $b\in \O$ with $bR=I$. Moreover, we write $I_1\sim_H I_2$ if for all $a_1,a_2\in R$ with $(a_1)=I_1,(a_2)=I_2$, we have $a_1\sim_H a_2$.   \\

\begin{center}
    \textit{Throughout this paper, $R$ is always a Dedekind domain with torsion class group and $\O\subseteq R$ is an order with conductor $\f$ such that $R/\f$ is finite and such that every class in $\Cl(R)$ contains a prime ideal, coprime to $\f$. }
\end{center} 

\smallskip

Using Proposition \ref{tkrull}, we can deduce the following basic facts about transfer Krull orders, which we will use freely from now on.

   \begin{lemma}\label{basic}
    Suppose that $\O$ is transfer Krull with Krull monoid $H\supseteq \O^\bullet$ as in Proposition \ref{tkrull}. Then \\
    
   \noindent (i) $R^\bullet\subseteq H$. \\

   \noindent (ii) $\mathcal{A}(\Reg(\O))\subseteq \mathcal{A}(R)$. \\

  \noindent  (iii) If $\O\subseteq R$ is a root extension, then $\mathcal{A}(\O)\subseteq \mathcal{A}(R)$.

   \end{lemma}

   \begin{proof}
  (i) Since $H$ is a completely integrally closed overmonoid of $H$, we obtain $\widehat{\O^\bullet}=R^\bullet \subseteq H$. \\

 \noindent (ii) Let $a\in \mathcal{A}(\Reg(\O))$ and suppose that we can write $a=bc$ for non-units $b,c\in \Reg_\f(R)$. Since $a$ is an $\O$-atom, it is an $H$-atom by Proposition \ref{tkrull}. By (i), we have $b,c\in H$ and thus $b\in H^\times$ or $c\in H^\times$. Suppose w.l.o.g. that $b\in H^\times$ and let $n\in\mathbb{N}$ such that $b^n\in \O$. Then $b^n\in H^\times\cap \O=\O^\times$, a contradiction. \\

 \noindent (iii) See the proof of (ii).
   \end{proof}

\section{Regular elements of transfer Krull orders}
\noindent The aim of this section is to prove that $\Pic(\O)\simeq\Cl(R)$ for a transfer Krull order $\O$, provided that $|\Cl(R)|\neq 2$. This establishes a strong connection between the regular elements of $\O$ and the monoid $\Reg_\f(R)$.  \\
\begin{center}
    \textit{Throughout this section, we suppose in addition to our initial assumptions that $\O$ is transfer Krull with Krull monoid $H\supseteq \O^\bullet $ as in Proposition \ref{tkrull}.} 
\end{center}

\begin{lemma}\label{abs}
    Let $a\in \Reg_\f(R)$ be absolutely irreducible in $R$. Then $a\sim \O$.
\end{lemma}
\begin{proof}
    Let $n\geq 2$ such that $a^n\in \O$. By the transfer property, there is $b\in \O$ such that $b$ divides $a^n$ in $\O$ and $a\sim_H b$. Since $a$ is absolutely irreducible in $R$, we have $b\sim_R a^k$ for some $k\leq n$ and consequently $a\sim_H a^k$. If $k\neq 1$, then $a\in H^\times$ and consequently $a^n\in H^\times\cap\O=\O^\times$, a contradiction. Hence $k=1$ and $a\sim \O$.
\end{proof}

Before continuing, we recall a simple group-theoretical fact. If $G$ is an (additive) abelian torsion group, $|G|\neq 2$ and $g\in G$ with $\ord(g)=2$, then there is $h\neq g$ such that $\ord(h)=2$ or $\ord(h)=2k$ for a positive integer $k\geq 2$ and $kh=g$. 

\begin{lemma}\label{regp}
   Suppose that $|\Cl(R)|\neq 2$ and let $P_1,P_2,P_3\in \spec_\f(R)$. \\

   \label{regp:inv} (i) If $[P_1]=-[P_2]$, then $P_1P_2\sim \O$. \\

  \label{regp:ord2}  (ii) If $\ord([P_1])=\ord([P_2])=2$, $[P_1]\neq [P_2]$ and $[P_3]=[P_1P_2]$, then $P_1P_2P_3\sim \O$. \\

   \label{regp:even} (iii) If $\ord([P_1])=2, \ord([P_2])=2k$ for some $k\geq 2$ such that $k[P_2]=[P_1]$, then $P_1P_2^k\sim \O$. \\

\end{lemma}
\begin{proof}

    (i) We will first prove the assertion for all $P_1\in \spec_\f(R)$ for which $n:=\ord([P_1])\neq 2$. If $n=1$, the claim follows immediately from Lemma \ref{abs}. Hence we can assume that $n\geq 3$. Let $a\in \Reg_\f(R)$ such that $aR=P_1P_2$. Since $a^n$ is a product of absolutely irreducibles in $R$, we obtain $a^n\sim\O$ by Lemma \ref{abs} and we find $b\in\O$ with $a^n\sim_R b$. Consequently, the transfer property yields $\mathsf{L}_H(a^n)=\mathsf{L}_\O(b)$. Our claim is now proved by using the fact that $\max\mathsf{L}_H(a^n)\geq n$ and by observing that, since neither $P_1^2$ nor $P_2^2$ are principal, $b$ has an $\O$-factorization of length $n$ only if $P_1P_2\sim \O$. \\

    Before continuing with the case $n=2$, we will first prove the other two statements. \\

  \noindent (ii) Let $a\in \Reg_\f(R)$ such that $aR=P_1P_2P_3$. Since $a^2$ is a product of absolutely irreducibles, Lemma \ref{abs} yields $a^2\sim \O$ and we find $b\in \O$ such that $b\sim_R a^2$. By the transfer property, we obtain $c,d\in \O\setminus\O^\times$ such that $cd=b$ and $c\sim_H a\sim_H d$. If $P_1P_2P_3 \not\sim \O$, then w.l.o.g., $cR=P_{\sigma(1)}^2$ and $dR=P_{\sigma(2)}^2P_{\sigma(3)}^2$, where $\sigma:[1,3]\to [1,3]$ is a permutation. But then $c\in \mathcal{A}(\O)\subseteq \mathcal{A}(H)$ and $d\not\in\mathcal{A}(H)$, a contradiction. \\

   \noindent (iii) Let $P_3\in \spec_\f(R)$ with the property that $[P_2]=-[P_3]$, let $a,b\in \Reg_\f(R)$ such that $aR=P_1P_2^k$ and $bR=P_1P_3^k$ and suppose that $a\not \sim \O$. We claim that also $b\not \sim \O$. Otherwise, we may assume that $b\in \O$ and observe that $P_1^2\sim\O$ by Lemma \ref{abs} and, since $\ord([P_2])>2$, that $P_2P_3\sim\O$ by (i).
   Hence \[abR=P_1^2P_2^kP_3^k\sim\O\] and, after possibly replacing $a$ with an $R$-associate, we obtain $ab\in \O$. However, since $b\in \Reg(\O)$, this yields $a\in \O$, a contradiction. Hence $P_1P_2^k\not\sim \O$ and $P_1P_3^k\not\sim \O$.\\
   
   Since $a^2$ is a product of absolutely irreducibles in $R$, Lemma \ref{abs} yields an element $c\in \O$ such that $a^2\sim_R c$. By the transfer property, there are $d,e\in \O$ such that $de=c$ and $d\sim_H a\sim_H e$. Clearly, either $d$ or $e$ must generate the principal ideal $P_1^2$ and we obtain \[P_1^2\sim_H P_1P_2^k.\] Repeating the same argument for $P_3$ in place of $P_2$, we obtain \[P_1^2\sim_H P_1P_3^k\] and consequently \[P_1^4\sim_H P_1^2P_2^k P_3^k,\] which implies that \[P_1^2\sim_H P_2^kP_3^k.\] As we already proved, we have $P_1^2\sim \O$ and $P_2^kP_3^k\sim\O$. However, elements in $\O$ generating $P_1^2$ are atoms, while those that generate $P_2^kP_3^k$ are not atoms in $H$, a contradiction. \\
   
   We will now prove the remaining case in (i). Suppose that $\ord([P_1])=\ord([P_2])=2$ and let $a\in\Reg_\f(R)$ such that $aR=P_1P_2$. Since $|\Cl(R)|\neq 2$, there is $P_3\in\spec_\f(R)$ with the property that $\ord([P_3])=2$ and $[P_1]\neq[P_3]$ or such that $\ord([P_3])=2k$ for some $k\geq 2$ and $k[P_3]=[P_1]$. \\
   
   In the first case, let $P_4 \in\spec_\f(R)$ such that $[P_4]=[P_1P_3]$. By (ii), we have \[P_1P_3P_4\sim\O,\] \[P_2P_3P_4\sim\O\] and consequently \[P_1P_2P_3^2P_4^2\sim\O.\] By Lemma \ref{abs}, we find $b,c\in \Reg(\O)$ such that $bR=P_3^2$ and $cR=P_4^2$. After possibly replacing $a$ with an $R$-associate, we have $abc\in \O$ and since $bc\in \Reg(\O)$, we obtain $a\in \O$ and consequently $P_1P_2\sim \O$.  \\

   In the second case, we have $P_1P_3^k\sim\O$ and $P_2P_3^k\sim\O$ by (iii) and consequently \[P_1P_2P_3^{2k}\sim\O.\] By noting that $P_3^{2k}\sim \O$, we can proceed in the same manner as in the previous case to obtain $P_1P_2\sim\O$.

\end{proof}

\begin{lemma} \label{betaabs}
    Suppose that $|\Cl(R)|\neq 2$ and let $a,b\in \Reg_\f(R)$ such that $a\sim_H b$. If $a$ is absolutely irreducible in $R$, then $\boldsymbol{\beta}(a)=\boldsymbol{\beta}(b)$. 
\end{lemma}

\begin{proof}
   First, we will assume that $b$ is absolutely irreducible as well. Let $aR=P_1^n, bR=P_2^m$, let $Q_1\in\spec_\f(R)$ such that $[P_1]=-[Q_1]$ and suppose that $n>m$. By Lemmas \ref{abs} and \ref{regp}, there are $c,d\in \O$ such that $cR=P_1Q_1$ and $dR=P_2^mQ_1^n$. By assumption, we have $c^n\sim_H d$ and thus $\mathsf{L}_\O(c^n)=\mathsf{L}_\O(d)$. Since $n>m\geq 1$, we have $c\in\mathcal{A}(\O)$ and thus $n\in\mathsf{L}_\O(c^n)$. It is then easy to see that $n\in\mathsf{L}_\O(d)$ is only possible if $P_2$ is principal and $n=2$. Thus we will suppose that $n=2$ and $m=1$.\\

Since $|\Cl(R)|\neq 2$, there is $P_3\in\spec_\f(R)$ with the property that $\ord([P_3])=2$ and $[P_1]\neq[P_3]$ or such that $\ord([P_3])=2k$ for some $k\geq 2$ and $k[P_3]=[P_1]$. \\

In the first case, let $P_4 \in\spec_\f(R)$ such that $[P_4]=[P_1P_3]$. By Lemmas \ref{abs} and \ref{regp}, there are $c,d\in \O$ such that $cR=P_1P_3P_4$ and $dR=P_2P_3^2P_4^2$. By assumption, we have $c^2\sim_H d$, since $P_1^2\sim_H P_2$. Moreover, we have $c\in\mathcal{A}(\O)$ and thus $2\in\mathsf{L}_\O(c^2)$. However, any $\O$-factorization of $d$ of length two yields an $\O$-atom, that is not an $R$-atom, contradicting Lemma \ref{basic}. \\

In the second case, let $P_4 \in\spec_\f(R)$ such that $[P_3]=-[P_4]$. By Lemmas \ref{abs} and \ref{regp}, we find $c,d,e\in \O$ such that $cR=P_1P_3^k$, $dR= P_1P_4^k$ and $eR=P_2^2P_3^{k}P_4^k$. We then have $cd\sim_H e$ and $2\in \mathsf{L}_\O(cd)$, since $c,d\in\mathcal{A}(\O)$. Recalling that $k\geq 2$, it is easy to see that $2\in\mathsf{L}_\O(e)$ implies the existence of an $\O$-atom that is not an $R$-atom, contradicting Lemma \ref{basic} again. \\

Hence we have shown that $n=m$. We will now prove that $[P_1]=[P_2]$. This is clearly true if $n=1$. Suppose now that $n\geq 2$ and let $Q_1\in\spec_\f(R)$ such that $[P_1]=-[Q_1]$. Again, by Lemmas \ref{abs} and \ref{regp}, we find $c,d\in\O$ such that $cR=P_1Q_1$ and $dR=P_2^nQ_1^n$. Then $c^n\sim_H d$ and $n\in\mathsf{L}_\O(c^n)$. If $[P_1]\neq[P_2]$, then $P_2Q_1$ is not principal, which clearly implies that $n=2$. \\

Let $P_3\in\spec_\f(R)$ such that $[P_3]=[P_1P_2]$. Again, by Lemmas \ref{abs} and \ref{regp}, we find elements $e,f\in\O$ such that $eR=P_1P_2P_3$ and $fR=P_1^4P_3^2$. Then we have $e^2\sim_H f$ and $2\in \mathsf{L}_\O(e^2)=\mathsf{L}_\O(f)$, which contradicts Lemma \ref{basic}. Hence our claim is proved, when $b$ is absolutely irreducible in $R$. \\

Let now $aR=P^n$ and $bR=P_1^{n_1}\cdots P_k^{n_k}$ for prime ideals $P,P_1,\ldots, P_k\in \spec_\f(R)$ and positive integers $n, n_1,\ldots,n_k$, where $n=\ord([P])$ and suppose that $a\sim_H b$. There is a positive integer $m$ such that $a^m,b^m\in\O$ and such that $b^m$ is a product of absolutely irreducible elements in $R$. By the transfer property, there is $c\in\O$, dividing $a^m$, such that $c\sim_H d$, where $dR=P_i^{\ord([P_i])}$ and $i\in[1,k]$. By Lemma \ref{abs}, we can assume that $d\in\mathcal{A}(\O)$. Hence $c\in\mathcal{A}(\O)$ and $a\sim_R c$. By the first part of the proof, we obtain $\ord([P_i])=n$ and $[P_i]=[P]$. The only thing left to show is that $n=\sum_{i=1}^{k}n_i$. This follows from the fact that $a$ is an $R$-associate of an $\O$-atom, whence $b$ is an $R$-atom by Lemma \ref{basic}.
\end{proof}

We can now extend this result for arbitrary elements of $\Reg_\f(R)$.

\begin{lemma}\label{beta}
    Suppose that $|\Cl(R)|\neq 2$ and let $a,b\in \Reg_\f(R)$ such that $a\sim_H b$. Then $\boldsymbol{\beta}(a)=\boldsymbol{\beta}(b)$.
\end{lemma}
\begin{proof}
   There is a positive integer $m$ such that $a^m\in\O$ and such that $b^m$ can be written as a product of absolutely irreducible elements in $R$. Clearly, we have $\boldsymbol{\beta}(a)=\boldsymbol{\beta}(b)$ if and only if $\boldsymbol{\beta}(a^m)=\boldsymbol{\beta}(b^m)$ and thus we can assume that $m=1$. By Lemma \ref{abs}, there are $b_1,\ldots,b_k\in\O$, absolutely irreducible elements in $R$ such that $\prod_{i=1}^{k}b_i\sim_R b$. We proceed by induction on the positive integer $k$. If $k=1$, we are done by Lemma \ref{betaabs}. Otherwise, since we have $a\sim_H \prod_{i=1}^{k}b_i$, the transfer property implies the existence of elements $c,d\in\O$ such that $cd=a$, $c\sim_H b_k$ and $d\sim_H \prod_{i=1}^{k-1}b_i$. By the induction hypothesis and Lemma \ref{betaabs}, the claim is then proved.
\end{proof}

\begin{proposition}\label{pic}
    Suppose that $|\Cl(R)|\neq 2$. Then $\Cl(R)\simeq \Pic(\O)$.
\end{proposition} 

\begin{proof}
    It is enough to show that $a\sim\O$ for every element $a\in\mathcal{A}(\Reg_\f(R))$. Let $aR=P_1^{n_1}\cdots P_k^{n_k}$ for pairwise distinct prime ideals $P_1,\ldots,P_k\in\spec_\f(R)$ and positive integers $n_1,\ldots,n_k$. By Lemmas \ref{abs} and \ref{regp}, we can assume that $\sum_{i=1}^{k}n_i\geq 3$. Let now $Q_1,\ldots, Q_k\in\spec_\f(R)$ such that $[Q_i]=-[P_i]$ and $Q_i=P_i$ if $\ord([P_i])=2$ for every $i\in[1,k]$. Moreover, let $b\in\Reg_\f(R)$ such that $bR=Q_1^{n_1}\cdots Q_k^{n_k}$. By Lemma \ref{regp}, we have $P_iQ_i\sim\O$ for every $i\in[1,k]$ and thus we can find $c\in\O$ with the property $c\sim_R ab$. By the transfer property, there is $d\in\O$, dividing $c$ such that $d\sim_H a$. By Lemma \ref{beta}, we then have $\boldsymbol{\beta}(a)=\boldsymbol{\beta}(d)$. \\
    
    We will show that for every $i\in[1,k]$, $P_i^{n_i}$ divides $dR$, which, together with the fact that $\boldsymbol{\beta}(a)=\boldsymbol{\beta}(d)$, implies that $a\sim_R d$ and proves our claim. Suppose that $\ord([P_i])=2$. Then $n_i=1$ and $[P_i]\neq [P_j]$ for all $j\neq i$, since we are assuming that $\sum_{i=1}^{k}n_i\geq 3$. Similarly, we have $[P_i]\neq[Q_j]$ for all $j\neq i$. Then $\boldsymbol{\beta}(a)=\boldsymbol{\beta}(d)$ implies that $P_i$ or $Q_i$ divides $dR$. However, we have $P_i=Q_i$ by construction. Suppose now that $\ord([P_i])\geq 3$ (note that $\ord([P_i])=1$ is not possible since $a$ is an $R$-atom and $\sum_{i=1}^{k}n_i\geq 3$) and that $P_i^{n_i}$ does not divide $dR$. Then there is a $Q_j$, dividing $dR$, such that $[Q_j]=[P_i]$. We clearly have $i\neq j$, whence $P_iP_j$ is principal, again contradicting the fact that $\sum_{i=1}^{k}n_i\geq 3$.
\end{proof}

\section{Localizations of transfer Krull orders}

\noindent In this section, we use Proposition \ref{pic} to study the irregular prime ideals and localizations of a transfer Krull order $\O$ before proving our main theorem.  

\begin{proposition}\label{root2}
  Suppose that $\O$ is transfer Krull with Krull monoid $H\supseteq \O^\bullet$ as in Proposition \ref{tkrull} and that $\Pic(\O)\simeq\Cl(R)$. Then $\O\subseteq R$ is a root extension.
\end{proposition}

\begin{proof}
    By \cite[Theorem 2]{Ka20a}, we need to show that the Spec-map is bijective. If this is not the case, then there is $\p\in \spec(\O)$ with $\f\subseteq \p$ with at least two $R$-prime ideals lying over $\p$. Let $P_1,\ldots, P_s$ denote these prime ideals, where $s\geq 2$ and let $n_i=\text{ord}([P_i])$ for each $i\in[1,s]$. Then $\O_\p^\bullet$ is a finitely primary monoid of rank $s$. Let $p_1,\ldots,p_s$ 
    be pairwise non-associate prime elements of $\overline{\O_\p}$ such that $p_i\in P_i$ for all $i\in[1,s]$. Moreover, for each $i\in[1,s]$, let $k_i\in\mathbb{N}$ such that $k_i>n_i$. \\

    By Lemma \ref{finp}, there is $x\in \mathcal{A}(\O_\p)$ with $\mathsf{v}_{p_i}(x)\geq k_i$ for every $i\in[1,s]$. If \[P_1^{\mathsf{v}_{p_1}(x)}\cdots P_s^{\mathsf{v}_{p_s}(x)}\] is not principal, take $Q\in\spec_\f(R)$ such that \[QP_1^{\mathsf{v}_{p_1}(x)}\cdots P_s^{\mathsf{v}_{p_s}(x)}\] is principal. Otherwise, set $Q=R$. Moreover, let $a\in R$ be such that \[aR=QP_1^{\mathsf{v}_{p_1}(x)}\cdots P_s^{\mathsf{v}_{p_s}(x)}.\]

     Since $\Pic(\O)\simeq\Cl(R)$, we can assume that $a\in\O$ and $a\sim_{\O_\p}x$. Since $x$ is an atom of $\O_\p$, we obtain that $a\in\mathcal{A}(\O)\subseteq \mathcal{A}(H)$. Let $b_i\in R$ such that $b_iR=P_i^{n_i}$. Since we have $\mathsf{v}_{p_i}(x)> n_i$ by assumption, we can write $a=b_ic$, where $c\in R$ is contained in every prime ideal $P_1,\ldots, P_s$. Then either $b_i$ or $c$ is a unit in $H$. However, $c\not\in H^\times$, since $c^n\in\O$ for some $n\in\mathbb{N}$, which implies that $b_i\in H^\times$. \\

     Hence we find $b_1,\ldots,b_s\in R\cap H^\times$ such that $b_iR=P_i^{n_i}$ for all $i\in [1,s]$. However, there exists $n\in\mathbb{N}$ such that ${(b_1\cdots b_s)}^n \in \O \cap H^\times,$ which is a contradiction. 
    
\end{proof}

\begin{proposition} \label{value}
  Suppose that $\O$ is transfer Krull with Krull monoid $H\supseteq \O^\bullet$ as in Proposition \ref{tkrull} and that $\Cl(R)\simeq\Pic(\O)$. Let $P\in\spec(R)$ with $n=\ord([P])$ such that $P$ divides $\f$ and let $\p=P\cap\O$. Then $\mathsf{v}_\p(\mathcal{A}(\O_\p))\subseteq \{1,2\}$, if $n=2$ and $\mathsf{v}_\p(\mathcal{A}(\O_\p))= \{1\}$ otherwise. In particular, we have $\O\cdot R^\times=R$.
\end{proposition}

     \begin{proof}
        Suppose first that there is  $x\in \mathcal{A}(\O_\p)$ with $\mathsf{v}_\p(x)>n$. If $P^{\mathsf{v}_\p(x)}$ is not principal, let $Q\in\spec_\f(R)$ such that $QP^{\mathsf{v}_\p(x)}$ is principal. Otherwise set $Q=R$. Since $\Cl(R)\simeq \Pic(\O)$, we find $a\in\O$ with $aR=QP^{\mathsf{v}_\p(x)}$ such that $a\sim_{\O_\p}x$. Then, since $\mathsf{v}_\p(x)>n$, we see that $a$ is an $\O$-atom but not an $R$-atom, a contradiction. This proves the claim for $n\leq 2$. \\

        Suppose now that $n\geq 3$ and that there is $x\in\mathcal{A}(\O_\p)$ with $\mathsf{v}_\p(x)=m$, where $1<m<n$. Similar to above, we find $a\in\O$ and $Q_1,Q_2\in\spec_\f(R)$ satisfying $aR=Q_1Q_2P^m$, $[Q_1]=-[P]$, $[Q_2]=-[P^{m-1}]$ and $a\sim_{\O_\p}x$. Since $a$ is not an $R$-atom, it cannot be an $\O$-atom, whence there are non-units $b,c\in \O$ such that $a=bc$. However, both $b$ and $c$ are contained in $P$ and thus both $b$ and $c$ are proper divisors of $x$ in $\O_\p$, a contradiction. \\
   
        Suppose now that there is $x\in\mathcal{A}(\O_\p)$ with $\mathsf{v}_\p(x)=n$. Take $Q\in\spec_\f(R)$ with $[Q]=-[P]$ and $a\in\O$ such that $aR=Q^nP^n$ and $a\sim_{\O_\p}x$. Since $a$ is not an $R$-atom, it is not an $\O$-atom, whence there is $b\in\O$, properly dividing $a$. Note that, if $b$ is contained in $P$, then $bR=P^n$, since $x\in\mathcal{A}(\O_p)$. Otherwise we have $bR=Q^n$. Let $d,e\in R$, satisfying $dR=QP$ and $eR=Q^{n-1}P^{n-1}$.  Then, by the transfer property, there are $b,c\in\O$ such that $a=bc, b\sim_H d$ and $c\sim_H e$. However, $c$ is an $\O$-atom, but $e$ is not an $R$-atom, a contradiction. \\

        Hence $\mathsf{v}_\p(\mathcal{A}(\O_\p))\subseteq \{1,2\}$, if $n=2$ and $\mathsf{v}_\p(\mathcal{A}(\O_\p))= \{1\}$ otherwise. Moreover, we have $1\in\mathsf{v}_\p(\mathcal{A}(\O_\p))$ for all $\p\in\spec(\O)$ (note that $\mathsf{v}_\p(\mathcal{A}(\O_\p))=\{2\}$ is impossible since $\O_\p^\bullet$ is a finitely primary monoid of finite exponent), whence it follows that $\O\cdot R^\times=R$.
     \end{proof}

\begin{proof}[Proof of Theorem 1] (1) Suppose that $|\Cl(R)|\neq 2$. Let $P_1,\ldots,P_k\in\spec(R)$ be the prime ideals that divide $\f$, let $\p_i=P_i\cap\O$ and let $p_i$ be a prime element of $\overline{\O_{\p_i}}$ for every $i\in[1,k]$. If $\O$ is transfer Krull, we obtain $\Cl(R)\simeq \Pic(\O)$ by Proposition \ref{pic}. Then, by Proposition \ref{value}, we have  $\O\cdot R^\times=R$ and $\mathsf{v}_{\p_i}(\mathcal{A}(\O_{\p_i}))=\{1\}$ for every $i\in[1,k]$ such that $\ord([P_i])\neq 2$. \\ 

    Suppose now that $\ord([P_i])=2$ for some $i\in[1,k]$. Then $\mathsf{v}_{\p_i}(\mathcal{A}(\O_{\p_i}))\subseteq \{1,2\}$ by Proposition \ref{value}. Write $P=P_i$ and $\p=\p_i$ and suppose that there is $x\in\mathcal{A}(\O_\p)$ with $\mathsf{v}_\p(x)=2$. Then we find non-principal prime ideals $Q_1,Q_2,Q_3,Q_4\in\spec_\f(R)$ such that $[Q_1]\neq [P]$, $[Q_1]=-[Q_2],[Q_3]=[Q_1P]$ and $[Q_3]=-[Q_4]$ and an element $a\in\O$ with $aR=Q_1Q_2Q_3Q_4P^2$ such that $a\sim_{\O_\p}x$. By Proposition \ref{value}, we have $\O\cdot R^\times=R$, whence there are $b,c\in\O$ with $bR=Q_2Q_3P$ and $cR=Q_1Q_4P$. We clearly have $b,c\in\mathcal{A}(\O)\subseteq \mathcal{A}(H)$ and consequently $2\in\mathsf{L}_H(bc)=\mathsf{L}_H(a)$. However, if $d,e\in \mathcal{A}(\O)\subseteq \mathcal{A}(R)$ such that $a=de$, then $d$ and $e$ are both contained in $P$, which contradicts the fact that $x\in\mathcal{A}(\O_\p)$.  \\

    Hence we have shown that $\mathsf{v}_{\p_i}(\mathcal{A}(\O_{\p_i}))=\{1\}$ for all $i\in[1,k]$. \\

     To prove the converse, we will show that the inclusion $\O^\bullet\hookrightarrow R^\bullet$ is a transfer homomorphism. By assumption, we have $\O\cdot R^\times=R$ and also $\O\cap R^\times=\O^\times$. Thus the only thing left to show is that if $a,b\in R$  with $ab\in\O$, then there is $\varepsilon\in R^\times$ such that $\varepsilon a,\varepsilon^{-1}b \in\O$. \\

     For $i\in[1,k]$, let $a_i=\mathsf{v}_{P_i}(a)$, $b_i=\mathsf{v}_{P_i}(b)$ and let $\Phi:\O^\bullet\rightarrow \prod_{i=1}^{k}{{(\O_{\p_i}^\bullet)}_{\text{red}} }$ denote the canonical homomorphism. Since $\mathsf{v}_{\p_i}(\mathcal{A}(\O_{\p_i}))=\{1\},$ we find elements $x,y\in \prod_{i=1}^{k}{{(\O_{\p_i}^\bullet)}_{\text{red}} }$ such that $\mathsf{v}_{p_i}(x)=a_i$ and $\mathsf{v}_{p_i}(y)=b_i$ for all $i\in[1,k]$ and $\Phi(ab)=xy$. Moreover, we find $\varepsilon_1,\varepsilon_2\in R^\times$ such that $\varepsilon_1 a,\varepsilon_2 b\in\O$ and $\Phi(\varepsilon_1 a)=x,\Phi(\varepsilon_2 b)=y$. Then, since $ab\sim_R \varepsilon_1 a \cdot\varepsilon_2 b$ and $\Phi(ab)=\Phi(\varepsilon_1 a \cdot\varepsilon_2 b)$, we also have $ab\sim_\O \varepsilon_1 a \cdot\varepsilon_2 b$. Hence there is $\varepsilon\in R^\times$ such that $\varepsilon a\in \O$ and $ab= \varepsilon a\cdot \varepsilon_2 b$. This implies that $\varepsilon_2=\varepsilon^{-1}$ and proves the claim. \\

    \noindent (2) Suppose that $|\Cl(R)|=2$, which implies that $R$ is half-factorial. If $\O$ is transfer Krull, the claim follows from Proposition \ref{value}, once we prove that $\Cl(R)\simeq\Pic(\O)$. By Lemma \ref{abs}, we have $\pi\sim \O$ for every prime element $\pi\in \Reg_\f(R)$, which, by assumption implies that the kernel of the canonical surjection $\Pic(\O)\to\Cl(R)$ is trivial. \\

     For the converse, we claim that it is enough to show that $\mathcal{A}(\O)\subseteq \mathcal{A}(R)$. Indeed, for every $a\in\O$, we have $\mathsf{L}_\O(a)\subseteq \mathsf{L}_R(a)$ and since $R$ is half-factorial, we have $|\mathsf{L}_\O(a)|\leq |\mathsf{L}_R(a)|=1$. Hence $\O$ is half-factorial and thus transfer Krull. \\

     Let $a\in \mathcal{A}(\O)$ and suppose that we can write $a=bc$ for non-units $b,c\in R$. Then for all $\varepsilon_1,\varepsilon_2\in R^\times$ such that $\varepsilon_1 b,\varepsilon_2 c\in \O$, we have $a\not\sim_\O \varepsilon_1b\cdot\varepsilon_2c$.
     This implies that there is $P\in\spec(R)$ such that $a\not\sim_{\O_\p}\varepsilon_1b\cdot \varepsilon_2c$ for every such $\varepsilon_1,\varepsilon_2\in R^\times$, where $\p=P\cap\O$. Since we have $\O\cdot R^\times=R$, this implies that for every $x,y\in \O_{\p}$ with $\mathsf{v}_\p(x)=\mathsf{v}_P(b)$ and $\mathsf{v}_\p(y)=\mathsf{v}_P(c)$, we have $a\not\sim_{\O_\p}xy$.  
     This is clearly only possible if $2\in\mathsf{v}_\p(\mathcal{A}(\O_\p))$ (and thus $\mathsf{v}_\p(\mathcal{A}(\O_\p))=\{1,2\}$ since the finitely primary monoid $\O_\p^\bullet$ has finite exponent), $\mathsf{v}_P(a)$ is even and both $\mathsf{v}_P(b)$ and $\mathsf{v}_P(c)$ are odd. In particular, $P$ is not principal by assumption and we have $\mathsf{v}_P(a)\geq 2$. Hence we can write $a=de$ for elements $d,e\in R$, where $dR=P^2$. Since $a$ is not an $R$-atom, $e$ is a non-unit and for all $\varepsilon_1,\varepsilon_2\in R^\times$ such that $\varepsilon_1 d,\varepsilon_2 e\in \O$, we have $a\not\sim_\O \varepsilon_1d\cdot\varepsilon_2e$. However, as we have shown, this implies that $\mathsf{v}_Q(d)$ is odd for some $Q\in\spec(R)$, a contradiction.
\end{proof}

 \noindent \textbf{Remark} (1) Although it is unknown whether the additional assumption about the Picard group $\Pic(\O)$ is necessary for the second part of Theorem \ref{thm}, there is evidence that the case $|\Cl(R)|$=2 is indeed a special one. In \cite{Ra-23}, an example of a half-factorial order $\O$ in a cubic number field and a prime ideal $\p\in\spec(\O)$, for which $\mathsf{v}_\p(\mathcal{A}(\O_\p))=\{1,2\}$ can be found. \\

\noindent (2) If $\O$ is transfer Krull with $|\Cl(R)|=2$ and there is $\p\in\spec(\O)$ with $\mathsf{v}_\p(\mathcal{A}(\O_\p))=\{1,2\}$, then both $R$ and $\O$ are half-factorial, however, $\O^\bullet\hookrightarrow R^\bullet$ is not a transfer homomorphism. Indeed, let $P$ be the non-principal prime ideal lying over $\p$. Then we find a non-principal $Q\in\spec_\f(R)$ and an element $a\in \O$ with $aR=P^2Q^2$ such that $a$ is an atom in $\O_\p$. We can write $a=bc$, where $b,c\in R$ with $bR=cR=PQ$. Then, by the transfer property, we find $d,e\in\O$ such that $d\sim_R b, e\sim_R c$ and $a=de$. However, we have $\mathsf{v}_\p(d)=\mathsf{v}_\p(e)=1$, contradicting the fact that $a$ is an atom in $\O_\p$.

\providecommand{\bysame}{\leavevmode\hbox to3em{\hrulefill}\thinspace}
\providecommand{\MR}{\relax\ifhmode\unskip\space\fi MR }
\providecommand{\MRhref}[2]{%
  \href{http://www.ams.org/mathscinet-getitem?mr=#1}{#2}
}
\providecommand{\href}[2]{#2}


\begin{thebibliography}{10}

  \bibitem{Ba-Sm22a}
  N.R. Baeth and D.~Smertnig,, \emph{Lattices over {B}ass rings and graph agglomerations}, Algebras
    and {R}epresentation {T}heory \textbf{25} (2021), 669 -- 704.
  
  \bibitem{Ba-Re22a}
    A.~Bashir and A.~Reinhart, \emph{On transfer {K}rull monoids}, Semigroup Forum
      \textbf{105} (2022), 73 -- 95.
  
  \bibitem{B-B-N-S23a}
  J.P. Bell, K.~Brown, Z.~Nazemian, and D.~Smertnig, \emph{On noncommutative
    bounded factorization domains and prime rings}, J. Algebra \textbf{622}
    (2023), 404 -- 449.
  
  \bibitem{Br-Ge-Re20}
  J.~Brantner, A.~Geroldinger, and A.~Reinhart, \emph{On monoids of ideals of
    orders in quadratic number fields}, in Advances in Rings, Modules, and
    Factorizations, vol. 321, Springer, 2020, pp.~11 -- 54.
  
  \bibitem{Ch-Ge24a}
  G.W. Chang and A.~Geroldinger, \emph{On {D}edekind domains whose class groups
    are direct sums of cyclic groups}, J. Pure Appl. Algebra \textbf{228} (2024),
    Paper No. 107470, 14pp.

  
  \bibitem{Fa-Tr18a}
  Y.~Fan and S.~Tringali, \emph{Power monoids: {A} bridge between factorization
    theory and arithmetic combinatorics}, J. Algebra \textbf{512} (2018), 252 --
    294.
  
  \bibitem{Fr-Na-Ri19a}
  S.~Frisch, S.~Nakato, and R.~Rissner, \emph{Sets of lengths of factorizations
    of integer-valued polynomials on {D}edekind domains with finite residue
    fields}, J. Algebra \textbf{528} (2019), 231 -- 249.
  
  \bibitem{G-L-T-Z21}
  W.~Gao, C.~Liu, S.~Tringali, and Q.~Zhong, \emph{On half-factoriality of
    transfer {K}rull monoids}, Commun. Algebra \textbf{49} (2021), 409 -- 420.
  
  \bibitem{Ge-HK06a}
  A.~Geroldinger and F.~Halter-Koch, \emph{Non-{U}nique {F}actorizations.
    {A}lgebraic, {C}ombinatorial and {A}nalytic {T}heory}, Pure and Applied
    Mathematics, vol. 278, Chapman \& Hall/CRC, 2006.
  
  \bibitem{Ge-Ka-Re15a}
  A.~Geroldinger, F.~Kainrath, and A.~Reinhart, \emph{Arithmetic of seminormal
    weakly {K}rull monoids and domains}, J. Algebra \textbf{444} (2015), 201 --
    245.
  
  \bibitem{Ge-Ru09}
  A.~Geroldinger and I.~Ruzsa, \emph{Combinatorial {N}umber {T}heory and
    {A}dditive {G}roup {T}heory}, Advanced Courses in Mathematics - CRM
    Barcelona, Birkh{\"a}user, 2009.
  
  \bibitem{Ge-Sc-Zh17b}
  A.~Geroldinger, W.A. Schmid, and Q.~Zhong, \emph{Systems of sets of lengths:
    transfer {K}rull monoids versus weakly {K}rull monoids}, in Rings,
    Polynomials, and Modules, Springer, Cham, 2017, pp.~191 -- 235.
  
  \bibitem{Ge-Zh20a}
  A.~Geroldinger and Q.~Zhong, \emph{Factorization theory in commutative
    monoids}, Semigroup Forum \textbf{100} (2020), 22 -- 51.
  
  \bibitem{Gr22a}
  D.J.~Grynkiewicz, \emph{The {C}haracterization of {F}inite {E}lasticities:
    {F}actorization {T}heory in {K}rull {M}onoids via {C}onvex {G}eometry},
    Lecture Notes in Math., vol. 2316, Springer, 2022.
  
  
  \bibitem{HK20a}
  F.~Halter-Koch, \emph{An {I}nvitation to {A}lgebraic {N}umbers and {A}lgebraic
    {F}unctions}, CRC Press, Boca Raton, FL, 2020.
  
  \bibitem{HK22a}
  \bysame, \emph{Class {F}ield {T}heory and ${L}$-{F}unctions}, CRC Press, Boca
    Raton, FL, 2022.
  
  \bibitem{Ka05b}
  F.~Kainrath, \emph{On local half-factorial orders}, in Arithmetic {P}roperties
    of {C}ommutative {R}ings and {M}onoids, Lect. Notes Pure Appl. Math., vol.
    241, Chapman \& Hall/CRC, 2005, pp.~316 -- 324.
  
  \bibitem{Ka20a}
  \bysame, \emph{On some arithmetical properties of noetherian domains}, in
    Advances in {R}ings, {M}odules and {F}actorizations, Springer Proc. Math.
    Stat., vol. 321, Springer, 2020, pp.~217 -- 222.
  
  
  \bibitem{Ph12b}
  A.~Philipp, \emph{A precise result on the arithmetic of non-principal orders in
    algebraic number fields}, J. Algebra Appl. \textbf{11}, 1250087, 42pp.
  
  \bibitem{Pi00}
  M.~Picavet-L'Hermitte, \emph{Factorization in some orders with a {PID} as
    integral closure}, Algebraic {N}umber {T}heory and {D}iophantine {A}nalysis
    (F.~Halter-Koch and R.~Tichy, eds.), Walter de Gruyter, 2000, pp.~365 -- 390.
  
  \bibitem{Po-24}
  P. Pollack, \emph{Half-factorial real quadratic orders}, Arch. Math. (Basel) \textbf{122} (2024), 491 -- 500.
  
  \bibitem{Po-24a}
  \bysame, \emph{Maximally elastic quadratic fields}, J. Number Theory \textbf{267} (2025),  80 -- 100.
  
  \bibitem{Ra-23}
  B.~Rago, \emph{A characterization of half-factorial orders in algebraic number
    fields}, Acta Arith., to appear (2024).
  
  \bibitem{Re23a}
  A.~Reinhart, \emph{On orders in quadratic number fields with unusual sets of
    distances}, Acta Arith. \textbf{211} (2023), 61 -- 92.
  
  \bibitem{Sc16a}
    W.A. Schmid, \emph{Some recent results and open problems on sets of lengths of
      {K}rull monoids with finite class group}, in {M}ultiplicative {I}deal
      {T}heory and {F}actorization {T}heory, Springer, 2016, pp.~323 -- 352.
  
  \bibitem{Sm13a}
  D.~Smertnig, \emph{Sets of lengths in maximal orders in central simple
    algebras}, J. Algebra \textbf{390} (2013), 1 -- 43.
  
  \bibitem{Sm19a}
  \bysame, \emph{{Factorizations in bounded hereditary noetherian prime rings}},
    Proc. Edinburgh Math. Soc. \textbf{62} (2019), 395 -- 442.
  \end{thebibliography}
\end{document}